\documentclass[a4paper]{article}
\usepackage[T1]{fontenc}
\usepackage[utf8]{inputenc}
\usepackage[babel]{csquotes}
\usepackage{microtype}
\usepackage{geometry} 
\usepackage{quoting} 
\quotingsetup{font=small}
\usepackage{amsmath}
\usepackage{amsfonts}
\usepackage{amssymb}
\usepackage{amsthm}
\makeatletter
\renewcommand{\pod}[1]{\allowbreak\mathchoice
	{\if@display \mkern 18mu\else \mkern 8mu\fi (#1)}
	{\if@display \mkern 18mu\else \mkern 8mu\fi (#1)}
	{\mkern4mu(#1)}
	{\mkern4mu(#1)}
}
\usepackage{mathrsfs}
\usepackage{mathtools}
\usepackage{braket}
\usepackage{bm}
\usepackage{enumerate}
\usepackage{empheq}
\usepackage{emptypage}
\usepackage{graphics}
\usepackage{tikz}
\usetikzlibrary{matrix,arrows,decorations.pathmorphing}
\usepackage{tikz-cd}
\usepackage{xparse}
\usepackage{stmaryrd}
\usepackage{latexsym,xspace,centernot}
\usepackage{pgf}
\usepackage{authblk}
\newcommand{\Z}{\mathbb{Z}}

\newcommand{\R}{\mathbb{R}}
\newcommand{\N}{\mathbb{N}}
\renewcommand{\epsilon}{\varepsilon}
\theoremstyle{plain}

\newtheorem{theorem}{Theorem}
\newtheorem{lemma}{Lemma}

\theoremstyle{remark}
\newtheorem{remark}{Remark}
\theoremstyle{definition}

\numberwithin{equation}{section}
\begin{document}
	\title{\textbf{Quadratic forms connected with Fourier coefficients of holomorphic and Maass cusp forms}}\date{}
	\author{Giamila Zaghloul}
	\affil{Dipartimento di Matematica, Università degli studi di Genova}
	\maketitle
	
	\begin{abstract}
		In this work we prove a prime number type theorem involving the normalised Fourier coefficients of holomorphic and Maass cusp forms, using the classical circle method. A key point is in a recent paper of Fouvry and Ganguly, based on  Hoffstein-Ramakrishnan's result about the non-existence of the Siegel zeros for $GL(2)$ $L$-functions, which allows us to improve preceding estimates.
	\end{abstract}
	\section{Introduction}
	The aim of this work is to prove a prime number type theorem involving the normalised Fourier coefficients of holomorphic and Maass cusp forms. Our result improves a recent bound obtained by Hu in \cite{hu} (which could have been reached with a simpler argument).\\
	The main tool used here is, as in Hu's paper, the classical circle method, but we give a different definition of the major arcs, as we shall see. Another key ingredient in our work is a result due to Fouvry and Ganguly (\cite{f-g}), based on Hoffstein-Ramakrishnan's result \cite[Section 5]{h-r} about the non-existence of the Siegel zeros for $GL(2)$ $L$-functions. Some of the estimates we need have been proved in detail by Hu, so we refer to \cite{hu} for the complete arguments. We first discuss a few preliminary results, underlining the differences with preceding estimates, and then we focus on the proof of the main theorem. \\

	In the following, we will assume $f$ to be either a holomorphic or a Maass cusp form, with normalised Fourier coefficients given by the sequence $(a(n))$, $n\geq 1$. We consider a definition of Maass forms general enough to include holomorphic cusp forms, in order to proceed with a unified argument (see \cite[Section 2.1]{f-g} for details). Moreover, we can suppose that $f$ is a \emph{Hecke-Maass} cusp form, that is an element of the so-called Hecke basis of the space of Maass forms. In particular, we will focus on Maass cusp forms for the full modular group $SL_2(\Z)$, which are trivially \emph{primitive forms} (again according to the notation of \cite[Section 2.3]{f-g}) and these include holomorphic cusp forms.\\
	We recall that for Maass forms the Ramanujan's conjecture, that is 
	\[
	|a(p)|\leq 2
	\]
	for $p$ prime, is still open, but we know that
	\[
	|a(p)|\leq 2p^{\theta}\quad\text{with}\quad\theta=\frac{7}{64}.
	\] 
	Hence, the coefficients of Maass cusp forms satisfy, for all $\epsilon>0$,
	\[
	a(n)=O(n^{\theta+\epsilon}),
	\]
	while it is well known that the bound 
	\[
	a(n)=O(n^{\epsilon})
	\]
	holds for holomorphic cusp forms (this result is due to Deligne). \\
	
		Moreover, let $\Lambda$ be the classical von Mangoldt function, 
		\[
		\Lambda(n)=\begin{cases}
		\log{p} \quad&\text{if}\quad n=p^m\quad \text{with}\quad m\geq 1\\
		0&\text{otherwise}
		\end{cases}
		\]
		and define
		\[ 
		\pi_{a,\Lambda}(x)=\sum_{m^2_1+m^2_2+m^2_3\leq x}a(m^2_1+m^2_2+m^2_3)\Lambda(m^2_1+m^2_2+m^2_3).
		\]
		In his paper, Hu proves a bound of the form
		\begin{equation}\label{1.1}
		\pi_{a,\Lambda}(x)=O(x^{3/2}\log^c x),
		\end{equation}
		where $c$ is a suitable positive constant. We claim that, taking into account the absence of exceptional zeros proved by Hoffstein-Ramakrishnan and a consequent estimate of \cite{f-g}, a stronger bound can be obtained. More precisely, our purpose is to prove the following statement.
		\begin{theorem}\label{theorem 1}
			Let $a(n)$ be as above. There exists a constant $c>0$ such that 
			\[
			\pi_{a,\Lambda}(x)=O(x^{3/2}\exp(-c\sqrt{\log{x}})).
			\]
		\end{theorem}

		\section{Notation and outline of the method}
		Since we follow Hu's approach to the problem and mainly refer to \cite{hu} for the details, we introduce the same notation. For $\alpha\in \R$ and $y>1$, define
		\[
		S_1(\alpha,y)=\sum_{1\leq m\leq y}e(m^2\alpha),\qquad S_2(\alpha,y)=\sum_{|m|\leq y}e(m^2\alpha)
		\]
		and the exponential sum 
		\[T(\alpha,y)=\sum_{1\leq n\leq y}a(n)\Lambda(n)e(n\alpha).\]\\ \\
		As a first step, we have
		\[\pi_{a,\Lambda}(x)=\int_{0}^{1}S^3_2(\alpha,\sqrt{x})T(-\alpha,x)d\alpha.\] 
		Moreover, by $S_2(\alpha,y)=2S_1(\alpha,y)+1$, we get, for all $\epsilon>0$,
		\[
		\pi_{a,\Lambda}(x)=8\int_{0}^{1}S^3_1(\alpha,\sqrt{x})T(-\alpha,x)d\alpha+O(x^{1+\theta+\epsilon}).
		\]
		Let now $P=\exp(C\sqrt{\log x})$, where $C$ is a positive constant,
		and let $Q=xP^{-1}=x\exp(-C\sqrt{\log x})$. Thanks to the periodicity of the integrand, 
		\begin{equation}\label{2.1}
		\begin{split}
		\pi_{a,\Lambda}(x)&=8\int_{0}^{1}S^3_1(\alpha,\sqrt{x})T(-\alpha,x)d\alpha+O(x^{1+\theta+\epsilon})\\&=8\int_{1/Q}^{1+1/Q}S^3_1(\alpha,\sqrt{x})T(-\alpha,x)d\alpha+O(x^{1+\theta+\epsilon}). 
		\end{split}
		\end{equation}
		
		Dirichlet's theorem on rational approximation assures that for any real $\alpha$ and any $Q\geq 1$ there exists a rational number $a/q$, with $(a,q)=1$, $1\leq q\leq Q$ such that 
		\[
		\alpha=\frac{a}{q}+\beta\quad\text{with}\quad |\beta|\leq \frac{1}{qQ}.
		\]
		For $1\leq q\leq P$ and $1\leq a\leq q$, $(a,q)=1$ let 
		\[
		\emph{M}(a,q)=\bigg[\frac{a}{q}-\frac{1}{qQ},\frac{a}{q}+\frac{1}{qQ}\bigg].
		\]
		The set $\textit{M}$ of the \emph{major arcs} is given by the union of the above defined $\emph{M}(a,q)$, as $a,q$ run in the proper ranges. As usual, the set $\emph{m}$ of the \emph{minor arcs} is instead
		\[\textit{m}=\bigg[\frac{1}{Q},1+\frac{1}{Q}\bigg]\setminus\textit{M}.\] 
		Hence, by equation \eqref{2.1}
		\begin{equation}\label{2.2}
		\pi_{a,\Lambda}(x)=8\int_{\textit{M}}S^3_1(\alpha,\sqrt{x})T(-\alpha,x)d\alpha+8\int_{\textit{m}}S^3_1(\alpha,\sqrt{x})T(-\alpha,x)d\alpha+O(x^{1+\theta+\epsilon}).
		\end{equation}
		Our goal is now to give an estimate of the integrand on major and minor arcs. 
		
		\section{First estimates}
		In this section we first state a couple of results about $S_1(\alpha,\sqrt{x})$ on major and minor arcs. For the proof of these facts we refer to \cite{hu}, Lemma $4.1$ and $4.2$ respectively. Then, we will discuss a key result that will allow us to improve the estimate of $T(\alpha,x)$ on major arcs. \\
		
		We start introducing the Gauss sum. For $a,b\in \Z$, $q\in\N$ let 
		\[
		G(a,b,q)=\sum_{r=1}^{q}e\bigg(\frac{ar^2+br}{q}\bigg).
		\]
		Suppose that $q\geq 1$ and $(a,q)=1$, then the above sum satisfies (see \cite[Lemma 3.5]{hu})
		\[
		G(a,b,q)\ll \sqrt{q}.
		\]
		\begin{lemma}\label{lemma 1}
			\textbf{(Estimate of $S_1(\alpha,\sqrt{x})$ on major arcs)}\\
			Let $1\leq q\leq P$, $1\leq a\leq q$, $(a,q)=1$ and $\alpha=\frac{a}{q}+\beta\in M(a,q)$ with $|\beta|\leq \frac{1}{qQ}$, then 
			\[
			S_1(\alpha,\sqrt{x})=\frac{G(a,0,q)}{q}\sqrt{x}\int_{0}^{1}e(x\beta v^2)dv+O(\sqrt{q}\log(q+1)).
			\]
		\end{lemma}
		\begin{lemma}\label{lemma 2}
			\textbf{(Estimate of $S_1(\alpha,\sqrt{x})$ on minor arcs)}\\
			Let $\alpha=\frac{a}{q}+\beta\in \textit{m}$ with $1\leq a\leq q$, $(a,q)=1$ and $\beta\leq \frac{1}{qQ}$. 
			Hence,
			\[
			\begin{split}
			S_1(\alpha,\sqrt{x})&\ll x^{1/2}q^{-1/2}+q^{1/2}(\log q)^{1/2}+x^{1/4}(\log q)^{1/2}.
			\end{split}
			\]
		\end{lemma}
		\begin{remark}
			By Lemma \ref{lemma 2}, we have 
			\[
			S_1(\alpha,\sqrt{x})\ll x^{1/2}P^{-1/2}=x^{1/2}\exp(-\frac{C}{2}\sqrt{\log x}),
			\]
			recalling that, on minor arcs, $P\leq q\leq Q$ and the value of $P$.
		\end{remark}
		Let now $\chi$ be a Dirichlet character modulo $q$ and define 
		\[
		\psi_f(x,\chi)=\sum_{n\leq x}a(n)\chi(n)\Lambda(n).
		\]
		\begin{lemma}\label{lemma 3}
			There exists a constant $A>0$ such that 
				\[
				\psi_f(x,\chi)\ll \sqrt{q}x\exp(-A\sqrt{\log x}).
				\]
		\end{lemma}
		\begin{proof}
		We will not give the details of the proof. The bound can be easily deduced from \cite[Theorem 4.1]{f-g}. In particular, it follows with standard methods by formula $(28)$, which is essentially based on the zero-free region proved by Hoffstein-Ramakrishnan (see \cite[Theorem C, part (3)]{h-r}) and the consequent absence of exceptional zeros.   
 \end{proof}
		
		We can now give an estimate of $T(\alpha,x)$ on major arcs. 
		\begin{remark}In the below lemma, we will use the following identity. For $(a,m)=1$, 
		\begin{equation}\label{3.1}
		e\bigg(\frac{a}{m}\bigg)=\frac{1}{\varphi(m)}\sum_{\chi\pmod m}\bar{\chi}(a)\tau(\chi),
		\end{equation}
		where $\tau(\chi)$ is the Gauss sum defined as
		\begin{equation}\label{3.2}
		\tau(\chi)=\sum_{b\pmod m}\chi(b)e\bigg(\frac{b}{m}\bigg).
		\end{equation}
		Moreover, note that
		\begin{equation}\label{3.3}
		\chi(a)\tau(\bar{\chi})=\sum_{b\pmod m}\bar{\chi}(b)e\bigg(\frac{ab}{m}\bigg).
		\end{equation}
		For details, one could see \cite[Section 3.4]{i-k}. 
	\end{remark}
	\begin{lemma}\label{lemma 4}\textbf{(Estimate of $T(\alpha,x)$ on major arcs)}\\
		Let $1\leq q\leq P$, $1\leq a\leq q$, $(a,q)=1$ and $\alpha=\frac{a}{q}+\beta\in M(a,q)$ with $|\beta|\leq \frac{1}{qQ}$. Then, there exists a constant $B>0$ such that
		\[
		T(\alpha,x)\ll x\exp(-B\sqrt{\log x}).
		\]
	\end{lemma}
	\begin{proof} We apply equations \eqref{3.1}, \eqref{3.2}, \eqref{3.3}, partial summation and Lemma \ref{lemma 3}, getting
		\[
		\begin{split}
		& T(\alpha,x)=\sum_{n\leq x}a(n)\Lambda(n)e\bigg(\frac{an}{q}\bigg)e(n\beta)=\sum_{\substack{n\leq x\\(n,q)=1}}a(n)\Lambda(n)e\bigg(\frac{an}{q}\bigg)e(n\beta)+\sum_{\substack{n\leq x\\(n,q)>1}}a(n)\Lambda(n)e\bigg(\frac{an}{q}\bigg)e(n\beta)\\
		&=\sum_{\substack{n\leq x\\(n,q)=1}}a(n)\Lambda(n)\bigg(\frac{1}{\varphi(q)}\sum_{\chi\pmod q}\bar{\chi}(an)\tau(\chi)\bigg)e(n\beta)+O(x^{\theta+\epsilon}\log^2 x)\\
		&=\sum_{\substack{n\leq x\\(n,q)=1}}a(n)\Lambda(n)\bigg(\frac{1}{\varphi(q)}\sum_{\chi\pmod q}\bar{\chi}(n)\sum_{b\pmod q}\chi(b)e\bigg(\frac{ab}{q}\bigg)\bigg)e(n\beta)+O(x^{\theta+\epsilon}\log^2 x)\\
		&=\sum_{b\pmod q}e\bigg(\frac{ab}{q}\bigg)\bigg(\frac{1}{\varphi(q)}\sum_{\chi\pmod q}\chi(b)\sum_{n\leq x}a(n)\bar{\chi}(n)\Lambda(n)e(\beta n)\bigg)+O(x^{\theta+\varepsilon}\log^2 x)\\
		&=\sum_{b\pmod q}e\bigg(\frac{ab}{q}\bigg)\bigg(\frac{1}{\varphi(q)}\sum_{\chi\pmod q}\chi(b)\bigg(\psi_f(x,\bar{\chi})e(\beta x)-2\pi i\beta\int_{1}^{x}\psi_f(u,\bar{\chi})e(\beta u)du\bigg)\bigg)+O(x^{\theta+\varepsilon}\log^2 x)\\
		&\ll\sum_{b\pmod q}(1+|\beta|x)\sqrt{q}x\exp(-A\sqrt{\log x})\ll xP^{3/2}\exp(-A\sqrt{\log x})\ll x\exp(-B\sqrt{\log x}),
\end{split}
		\] 
		where $B=A-\frac{3}{2}C$ and $A$ is the constant of Lemma \ref{lemma 3}. Hence, as a first restriction on $C$ we impose 
		\begin{equation}\label{3.4}
		A-\frac{3}{2}C>0 \Longleftrightarrow C<\frac{2}{3}A.		\end{equation}
		Moreover, observe that the term $O(x^{\theta+\epsilon}\log^2 x)$ is due to the values of $n$ such that $(n,q)>1$. In fact
		\begin{equation}\label{3.5}
		\sum_{\substack{n\leq x\\(n,q)>1}}a(n)\Lambda(n)e(n\alpha)\ll \sum_{\substack{n\leq x\\(n,q)>1}}a(n)\Lambda(n)\ll x^{\theta+\epsilon}\sum_{\substack{n\leq x\\(n,q)>1}}\Lambda(n)\ll x^{\theta+\epsilon}\log^2x.
		\end{equation}
\end{proof}
	\begin{remark}
		In \cite{hu}, the author claims that, simply normalising the coefficients, it is possible to generalise \cite[Theorem 1]{per} to the case of Maass forms. However, the proof of that result is strongly based on Ramanujan's conjecture, so we are not sure of the estimate stated in \cite[Lemma 5.1]{hu}. For this reason, our argument for $T(\alpha,x)$ on major arcs follows a different approach.   
	\end{remark}
	\section{Proof of the main result}
	We are now ready to prove Theorem \ref{theorem 1}. We start considering the major arcs. By construction
	\[
	\int_{\textit{M}}S^3_1(\alpha,\sqrt{x})T(-\alpha,x)d\alpha=\sum_{1\leq q\leq P}\sum_{\substack{a=1\\(a,q)=1}}^{q}\int_{\frac{a}{q}-\frac{1}{qQ}}^{\frac{a}{q}+\frac{1}{qQ}}S^3_1(\alpha,\sqrt{x})T(-\alpha,x)d\alpha.
	\]
	Assume now, with the usual meaning, $\alpha\in \textit{M}(a,q)$. By Lemma \ref{lemma 1}, developing the cube and observing that a priori two of the terms dominate the others, 
	\[
	S^3_1(\alpha,\sqrt{x})\ll \frac{G^3(a,0,q)}{q^3}x^{3/2}+q^{3/2}\log^3(q+1)\ll x^{3/2}q^{-3/2}. 
	\]
	Hence, the bound obtained in Lemma \ref{lemma 4} gives
	\[
	S^3_1(\alpha,\sqrt{x})T(-\alpha,x)\ll x^{3/2}q^{-3/2}x\exp(-B\sqrt{\log x})=x^{5/2}q^{-3/2}\exp(-B\sqrt{\log x}). 
	\]
	Now, recalling that the length of the interval $\textit{M}(a,q)$ is $2(qQ)^{-1}$, 
	\[
	\int_{\frac{a}{q}-\frac{1}{qQ}}^{\frac{a}{q}+\frac{1}{qQ}}S^3_1(\alpha,\sqrt{x})T(-\alpha,x)d\alpha
	\ll (qQ)^{-1}x^{5/2}q^{-3/2}\exp(-B\sqrt{\log x})=q^{-1}x^{3/2}q^{-3/2}P\exp(-B\sqrt{\log x}),
	\]
	having $Q=xP^{-1}$, so $Q^{-1}=x^{-1}P$. We sum over the residue classes modulo $q$, obtaining
	\[
	\sum_{\substack{a=1\\(a,q)=1}}^{q}\int_{\frac{a}{q}-\frac{1}{qQ}}^{\frac{a}{q}+\frac{1}{qQ}}S^3_1(\alpha,\sqrt{x})T(-\alpha,x)d\alpha
	\ll x^{3/2}q^{-3/2}P\exp(-B\sqrt{\log x}),
	\]
 and finally over $1\leq q\leq P$. Observing that $\sum_{q=1}^{P}q^{-3/2}=O(1)$,
	\begin{equation}\label{4.1}
	\int_{M}S^3_1(\alpha,\sqrt{x})T(-\alpha,x)d\alpha\ll x^{3/2}P\exp(-B\sqrt{\log x})=x^{3/2}\exp(-(B-C)\sqrt{\log x}),
	\end{equation}
	since $P=\exp(C\sqrt{\log x})$. Note that, to our purpose we have to impose $B-C>0$, so
	\begin{equation}\label{4.2}
	C<B=A-\frac{3}{2}C\Longleftrightarrow C<\frac{2}{5}A.
	\end{equation}
	We still have to find an estimate for 
	\[
	\int_{\textit{m}}S^3_1(\alpha,\sqrt{x})T(-\alpha,x)d\alpha.
	\]
	Using Cauchy inequality, 
	\[
	\begin{split}
	\int_{\textit{m}}S^3_1(\alpha,\sqrt{x})T(-\alpha,x)d\alpha&\ll \max_{\textit{m}}|S_1(\alpha,\sqrt{x})|\int_{0}^{1}|S_1(\alpha,\sqrt{x})|^2|T(-\alpha,x)|d\alpha\\
	&\ll \max_{\textit{m}}|S_1(\alpha,\sqrt{x})|\bigg(\int_{0}^{1}|S_1(\alpha,\sqrt{x})|^4d\alpha\bigg)^{1/2}\bigg(\int_{0}^{1}|T(-\alpha,x)|^2d\alpha\bigg)^{1/2}.
	\end{split}
	\] 
	Now, as observed in Lemma \ref{lemma 2}, 
	\begin{equation}
	\max_{\textit{m}}|S_1(\alpha,\sqrt{x})|\ll x^{1/2}P^{-1/2}=x^{1/2}\exp(-\frac{C}{2}\sqrt{\log x}).
	\end{equation}
	Moreover, since for Maass forms the following average bound holds
	\[
	\sum_{n\leq x}|a(n)|^2\ll x,
	\]
	we have
	\begin{equation}
	\int_{0}^{1}|T(-\alpha,x)|^2d\alpha\ll \sum_{n\leq x}a^2(n)\Lambda^2(n)\ll x\log^2 x.
	\end{equation}
	Finally,
	\begin{equation}
	\begin{split}
	\int_{0}^{1}|S_1(\alpha,\sqrt{x})|^4d\alpha&=\int_{0}^{1} \sum_{1\leq m_1,m_2\leq \sqrt{x}}e((m^2_1-m^2_2)\alpha)\sum_{1\leq m_3,m_4\leq \sqrt{x}}e((m^2_3-m^2_4)\alpha)d\alpha\\
	&=\sum_{\substack{1\leq m_i\leq \sqrt{x}\\i=1,2,3,4}}\int_{0}^{1}e((m^2_1+m^2_3-m^2_2-m^2_4)\alpha)d\alpha\\&=\sum_{\substack{1\leq m_i\leq \sqrt{x}\\i=1,2,3,4\\m^2_1-m^2_2=m^2_4-m^2_3}}1\ll\sum_{n\leq x}d^2(n)\ll x\log^3x, 
	\end{split}
	\end{equation}
	where $d(n)$ is the number of divisors of $n$.
	Then, by the above estimates, 
	\begin{equation}\label{4.6}
	\begin{split}
	\int_{\textit{m}}S^3_1(\alpha,\sqrt{x})T(-\alpha,\sqrt{x})d\alpha&\ll x^{1/2}P^{-1/2}x^{1/2}\log x x^{1/2}\log^{3/2}x\\&=x^{3/2}\log^{5/2}x\exp(-\frac{C}{2}\sqrt{\log x})\ll x^{3/2}\exp(-C'\sqrt{\log x}),
	\end{split}
	\end{equation}
	where $C'$ is a positive constant such that $C'<\frac{C}{2}$.\\
	Now, equations \eqref{4.1} and \eqref{4.6} give
	\[
	\pi_{a,\Lambda}(x)\ll x^{3/2}\exp(-(B-C)\sqrt{\log x})+x^{3/2}\exp(-C'\sqrt{\log x}).
	\]
	Hence, combining \eqref{3.4} and \eqref{4.2}, we choose $P=\exp(C\sqrt{\log x})$, with $C<\frac{2}{5}A$, getting
	\[
	\pi_{a,\Lambda}(x)\ll x^{3/2}\exp(-c\sqrt{\log x}),
	\]
	where $c=\min(B-C, C')=\min(A-\frac{5}{2}C,C')$. \\
	
	To conclude, note that a similar argument can also be applied to improve the estimate in  \cite[Theorem 2]{per}. In fact, the statement can be reformulated as follows, according to the notation of \cite{per}.
	\begin{theorem}
		There exists a positive constant $c'$ such that
		\[
		r_{\tau}(N)=\sum_{n_1+n_2+n_3=N}\tau(n_1)\Lambda(n_1)\tau(n_2)\Lambda(n_2)\tau(n_3)\Lambda(n_3)\ll N^{37/2}\exp(-c'\sqrt{\log N}).
		\]
	\end{theorem}
	The key point is again Hoffstein-Ramakrishnan's result, since Ramanujan's $\tau$ function is a holomorphic cusp form of level $12$. This assures the non-existence of exceptional zeros and then an analogous version of Lemma \ref{lemma 3} even holds in this case. We can apply the circle method choosing $P=\exp(A'\sqrt{\log N})$, where $A'>0$ depends on the constant $A$ of Lemma \ref{lemma 3}. As a consequence, the Corollary in \cite{per} gives
	\[
	S_{\tau}(\alpha)\ll N^{13/2}\exp(-C\sqrt{\log N}),
	\]
	where again the positive constant $C$ depends on $A$ and 
	\[S_{\tau}(\alpha)=\sum_{n\leq N}\tau(n)\Lambda(n)e(n\alpha).\]
	The result then follows as in \cite{per} (the constant $c'$ of the statement will depend on $A$).


\begin{thebibliography}{99}
		
		\bibitem{f-g} E. Fouvry and S. Ganguly. \emph{Orthogonality between the M\"{o}bius function, additive
			characters, and Fourier coefficients of cusp forms}, Compos. Math. (5)150
		(2014), 763–797.
		
		\bibitem{h-r}  J. Hoffstein, D. Ramakrishnan, \emph{Siegel zeros and cusp forms}, Internat.
		Math. Res. Notices (1995), 279–308.
		
		\bibitem{hu} L. Hu, \emph{Quadratic forms connected with Fourier coefficients of Maass cusp forms}, Front. Math. China 2015, 1101-1112.
	\bibitem{i-k} H. Iwaniec, E. Kowalski, \emph{Analytic Number Theory}, AMS, Providence 2004.
		\bibitem{per}A. Perelli, \emph{On some exponential sums connected with Ramanujan's $\tau$ function}, Mathematika 31 (1984), 150-158.
	\end{thebibliography}
\end{document}